\documentclass[12pt, reqno]{amsart} 
\usepackage{amsmath, amssymb, amsthm} 
\usepackage{geometry} 
\usepackage{enumitem} 
\usepackage[backend=biber, style=numeric, sorting=nyt, doi=true, url=true]{biblatex}
\usepackage{xurl}
\usepackage{hyperref}
\hypersetup{
    colorlinks=true,
    linkcolor=blue,
    citecolor=blue,
    urlcolor=blue
}

\usepackage{tikz}
\usetikzlibrary{arrows.meta}

\AtBeginBibliography{\sloppy}

\theoremstyle{plain} 
\newtheorem{theorem}{Theorem}[section]
\newtheorem{lemma}[theorem]{Lemma}
\newtheorem{proposition}[theorem]{Proposition}

\newtheorem{conjecture}[theorem]{Conjecture}

\theoremstyle{definition} 
\newtheorem{definition}[theorem]{Definition}

\theoremstyle{remark} 
\newtheorem{remark}[theorem]{Remark}

\newcommand{\F}{\mathcal{F}}
\newcommand{\Conf}{\operatorname{Conf}}
\newcommand{\rk}{\operatorname{rk}}

\begin{document}

\title{Classification of analytic conformal product structures}

\author{Xianfeng JIANG}
\address{Laboratoire mathématique d'Orsay, Université Paris-Saclay, 91405 Orsay, France}
\email{xianfeng.jiang@universite-paris-saclay.fr}



\begin{abstract}
    We classify real-analytic conformal product structures on compact simply connected manifolds in all dimensions \(n\geq 2\). We prove that in dimension \(2\) and \(3\) no such structure exists, while in every dimension \(n\geq 4\) the two parallel distributions arise from a global product splitting.
\end{abstract}

\maketitle
\pagestyle{plain}

\section{Introduction}

The classical de Rham decomposition theorem states that a complete simply
connected Riemannian manifold with reducible holonomy splits globally as a
Riemannian product. In conformal geometry, one considers Weyl connections,
which preserve a conformal class rather than a fixed metric. A reducible Weyl
connection $D$ gives an orthogonal $D$-parallel splitting
$TM=T_1\oplus T_2$, called a \emph{conformal product structure}. Belgun and Moroianu \cite{BelgunMoroianu2011} studied these structures in connection with Weyl-parallel forms and reducible Weyl holonomy. Locally, a conformal product admits adapted product coordinates in which
\[
g=e^{2f_1}g_1+e^{2f_2}g_2;
\]
see also \cite[Proposition~2.3]{MoroianuPilca2025}. Unlike the Riemannian case, this local splitting need not extend to a global product.

Conformal products have also been studied under additional geometric assumptions. If $D$ is exact, it is the Levi-Civita connection of a metric in the conformal class. A closed, non-exact, non-flat reducible Weyl connection defines a \emph{locally conformally product} (LCP) structure. Matveev and Nikolayevsky \cite{MatveevNikolayevsky2015} constructed a compact example with such a connection. Kourganoff \cite{Kourganoff2019} proved a decomposition theorem for the universal covers of compact non-flat LCP manifolds, even though the associated Riemannian metric is incomplete. Flamencourt \cite{Flamencourt2024} studied adapted metrics and constructed further examples of LCP manifolds. These results concern closed Weyl connections, whereas here we do not assume that $D$ is closed.

There are also results for manifolds with special Riemannian metrics. Moroianu and Pilca classified conformal products on compact K\"ahler manifolds \cite{MoroianuPilca2025}. In particular, if $D$ is not the Levi-Civita connection of the given K\"ahler metric, then one of the two parallel distributions has rank one. For compact Einstein manifolds, they proved a warped-product result without assuming a global product splitting \cite{MoroianuPilcaEinstein2026}. For compact reducible Riemannian manifolds, they showed, under an additional assumption, that the metric is either conformally flat or locally a triple product \cite{MoroianuPilcaReducible2025}. The paper \cite{JiangConstantCurvature2026} proves non-existence results for compact manifolds of nonzero constant sectional curvature and certain compact irreducible locally symmetric spaces. Here we study global splitting without imposing curvature assumptions.

Compactness and simple connectedness do not suffice in the smooth case: in \cite{Jiang2026}, a non-global conformal product on $S^3$ is constructed. The main result of the present paper, Theorem~\ref{thm:analytic-main}, shows that real analyticity gives a de Rham-type result: {\em On a compact simply connected manifold of dimension at least four, every real-analytic conformal product structure admits a global product splitting compatible with $T_1$ and $T_2$.} In dimension three, no such analytic structure exists; the two-dimensional case is excluded by the topology of $S^2$. Note that this is a splitting of the manifold and the distributions, not necessarily a Riemannian product metric in the given conformal class.

The proof uses Tarquini's foliated Ferrand--Obata theorem \cite{Tarquini2004}, which shows that the induced analytic conformal foliation of codimension at least three is either Riemannian or transversely M\"obius. The Riemannian case is treated using the decomposition theorem of Ponge--Reckziegel \cite{PongeReckziegel1993}. The M\"obius case is handled by using a developing map along an integrable horizontal distribution. In the remaining four-dimensional $2+2$ case, a parallel Hermitian structure allows us to apply Beauville's theorem \cite{Beauville2000}; in dimension three, non-existence follows from Haefliger's theorem
\cite{Haefliger1962}.

Section~1 recalls the transverse structures and the relevant conjectures. Sections~2 and~3 treat the Riemannian and M\"obius cases, respectively, while Section~4 treats the low-dimensional cases.

\textbf{Acknowledgments.} The author would like to thank Andrei Moroianu for his guidance and support. The author is supported by a doctoral contract from Université Paris-Saclay, through the Programme blanc GS Mathématiques and the École Doctorale de Mathématiques Hadamard (EDMH).

\section{Preliminaries}

On a conformal manifold $(M,c)$, where $c$ is a conformal class of Riemannian metrics, a torsion-free linear connection $D$ compatible with the conformal structure $c$ is called a Weyl connection. Recall that by fixing a background metric $g$ in the conformal class $c$, there is a one-to-one correspondence between Weyl connections and 1-forms, i.e., for every Weyl connection $D$, there is a 1-form $\theta$ (called the Lee form), satisfying the formula $Dg=-2\theta \otimes g$, and conversely $D$ is uniquely determined by $\theta$.  

Following the terminology of \cite{Jiang2026}, a conformal product structure consists of a Weyl connection $D$ together with a $g$-orthogonal decomposition $TM=T_1\oplus T_2$ into two nonzero $D$-parallel distributions. Since $D$ is torsion-free, both $T_1$ and $T_2$ are integrable; we denote their induced foliations by $\F_1$ and $\F_2$, respectively.

We call the splitting a \emph{global product splitting} if there are manifolds
$M_1,M_2$ and a diffeomorphism $\Phi:M_1\times M_2\to M$
such that $d\Phi(TM_1)=T_1$ and $d\Phi(TM_2)=T_2$. This is a statement about the manifold and the two distributions;
it does not assert that $[g]$ contains a Riemannian product metric.

\begin{definition}\label{def:analytic-cps}
A conformal product structure $(M,[g],D,T_1,T_2)$ is called \emph{real-analytic} if $M$ is a real-analytic manifold, the conformal class $[g]$ admits local real-analytic representatives, $D$ is a real-analytic connection, and $T_1,T_2$ are real-analytic subbundles of $TM$. Here a distribution is called real analytic if it is locally spanned by real-analytic vector fields.
\end{definition}

After interchanging the two distributions if necessary, we may assume
$\dim T_2\geq\dim T_1$. 

\begin{definition}\label{def:conformal-foliation}
Let $\F$ be a foliation of codimension $q$ on $M$. Following \cite[Definition~1.1.1]{Tarquini2004}, $\F$ is called \emph{conformal} if there exist a
possibly disconnected $q$-dimensional Riemannian manifold $(N,h)$ together with a
foliation cover $\{U_i\}$, submersions $f_i:U_i\to N$ whose connected fibers are
the plaques of $\F$, and local diffeomorphisms $\gamma_{ij}$ between open subsets
of $N$ such that $f_i=\gamma_{ij}\circ f_j$ on overlaps and every $\gamma_{ij}$
is conformal with respect to $h$.
\end{definition}

\begin{definition}\label{def:Riemannian-foliation}
Following the standard terminology (see, for example, \cite{Molino1988}), a conformal foliation $\F$ represented by an $N$-cocycle $\{U_i,f_i,\gamma_{ij}\}$ is called \emph{Riemannian} if $N$ admits a Riemannian metric $h_N$ for which every transverse change of coordinates $\gamma_{ij}$ is a local isometry.
\end{definition}

A Riemannian metric $h_N$ on $N$ is called
\emph{holonomy-invariant} if
\[
\gamma_{ij}^*h_N=h_N
\]
for every transverse change of coordinates $\gamma_{ij}$. Equivalently, $h_N$ is holonomy-invariant if every element of the holonomy pseudogroup generated by the $\gamma_{ij}$ is a local isometry of $h_N$.

Our main result is the following:

\begin{theorem}[Analytic globalization]\label{thm:analytic-main} Let $(M,g)$ be compact, connected, and simply connected, with dimension $n\ge3$. Assume that $(M,[g],D,T_1,T_2)$ is a real-analytic conformal product structure. Then the splitting $TM=T_1\oplus T_2$ is a global product splitting.
\end{theorem}
\begin{remark}
This theorem also holds in case of $n=2$ trivially. Since $M$ is compact and simply connected, we have $M\simeq S^2$. If $TS^2=T_1\oplus T_2$, the line bundle $T_i$ must be orientable, thus there is a vector field generating $T_i$, which contradicts the Poincar\'e-Hopf Theorem.
\end{remark}

The main tool in the proof is the following theorem of Tarquini \cite[Theorem~0.0.1]{Tarquini2004}. We first introduce the required terminology.

\begin{definition}\label{def:transversely-analytic}
A conformal foliation $\mathcal{F}$ is called \emph{transversely real-analytic} if there are data $(N,h)$ and $\{U_i,f_i,\gamma_{ij}\}$ as in Definition~\ref{def:conformal-foliation} such that $h$ is a real-analytic Riemannian metric on $N$, and every transverse change of coordinates $\gamma_{ij}$ is real analytic.
\end{definition}

\begin{definition}\label{def:spherical-foliation}
Following the terminology in \cite[Example~1.1.3]{Tarquini2004}, a codimension-$q$ foliation is called a
$(\Conf(S^q),S^q)$-{\em foliation}, or {\em transversely M\"obius}, if it is represented by an $S^q$-cocycle
$\{U_i,f_i,\gamma_{ij}\}$ in which every $\gamma_{ij}$ is the restriction, to
an open subset of $S^q$, of an element of the conformal group $\Conf(S^q)$ of
the round sphere.
\end{definition}

\begin{theorem}[Tarquini]\label{thm:Tarquini}
Let $\F$ be a transversely real-analytic conformal foliation of codimension
$q\geq3$ on a compact, connected manifold. Then $\F$ is either Riemannian
or transversely M\"obius, equivalently, a $(\Conf(S^q),S^q)$-foliation.
\end{theorem}

Without loss of generality, we assume that $\dim T_1\le \dim T_2$. In the case of $\operatorname{codim}\mathcal{F}_1=\dim T_2\geq3$, in order to apply Theorem~\ref{thm:Tarquini}, we first show in the next two lemmas that $\mathcal{F}_1$ is a transversely real-analytic conformal foliation.

\begin{lemma}\label{lem:conformal-foliation}
The foliation $\F_1$ is a conformal foliation of codimension $q$. Its transverse
conformal structure is induced by $[g|_{T_2}]$.
\end{lemma}

\begin{proof}
By \cite[Proposition~2.3]{MoroianuPilca2025}, every point has an adapted product
neighborhood $U_i\simeq U_i^1\times U_i^2$ on which $T_1=TU_i^1$,
$T_2=TU_i^2$, and
\[
g=e^{2f_{1,i}}g_{1,i}+e^{2f_{2,i}}g_{2,i},
\]
where $g_{k,i}$ is a Riemannian metric on $U_i^k$. Let $\pi_i:U_i\to U_i^2$ be the second projection. On $T_2$ one has $g|_{T_2}=e^{2f_{2,i}}\pi_i^*g_{2,i}$. Hence $g_{2,i}$ represents the transverse conformal structure induced by $[g|_{T_2}]$.

On each connected component of the intersection $U_i\cap U_j$, the transverse projections satisfy
$\pi_i=\gamma_{ij}\circ\pi_j$ for a local diffeomorphism $\gamma_{ij}$ between
open subsets of the transverse factors. Since both $g_{2,i}$ and $g_{2,j}$
represent the same conformal class $[g|_{T_2}]$, we have
$\gamma_{ij}^*[g_{2,i}]=[g_{2,j}]$; thus $\gamma_{ij}$ is conformal. Taking
$N:=\coprod_i U_i^2$ and the metric whose restriction to $U_i^2$ is $g_{2,i}$,
the data $\{U_i,\pi_i,\gamma_{ij}\}$ form an $N$-cocycle as in
Definition~\ref{def:conformal-foliation}. Therefore $\F_1$ is conformal.
\end{proof}

\begin{lemma}\label{lem:analytic-foliation}
If the conformal product structure $(M,[g],D,T_1,T_2)$ is real analytic, then
$\F_1$ is a transversely real-analytic conformal foliation in the sense of
Definition~\ref{def:transversely-analytic}.
\end{lemma}

\begin{proof}
Since $T_1$ and $T_2$ are real-analytic complementary involutive distributions, the analytic Frobenius theorem gives real-analytic adapted product neighborhoods
\[
U_i=V_i\times N_i,
\qquad
\pi_i:U_i\longrightarrow N_i.
\]
Choose a real-analytic representative $g$ of $[g]$ on $U_i$. By the local structure theorem for
conformal product structures \cite[Proposition~2.3]{MoroianuPilca2025}, after
shrinking $U_i$ if necessary,
\[
g|_{T_2}=e^{2f_{2,i}}\pi_i^*g_{2,i}
\]
for a Riemannian metric $g_{2,i}$ on $N_i$.

Fix $x_i^0\in V_i$, let $s_i:N_i\rightarrow U_i$ be defined by $s_i(y):=(x_i^0,y)$, and set
\[
h_i:=s_i^*(g|_{T_2})=s_i^*\!\left(e^{2f_{2,i}}\pi_i^*g_{2,i}\right)=e^{2f_{2,i}(x_i^0,\cdot)}(\pi_i\circ s_i)^*g_{2,i}=e^{2f_{2,i}(x_i^0,\cdot)}g_{2,i}.
\]
Thus $h_i$ is real analytic and $[h_i]=[g_{2,i}]$. By Lemma~\ref{lem:conformal-foliation}, on each connected component of an overlap the transverse change of coordinates $\gamma_{ij}$ satisfies
\[
\pi_i=\gamma_{ij}\circ\pi_j,
\qquad
\gamma_{ij}^*[h_i]=[h_j].
\]
The maps $\gamma_{ij}$ are real analytic because the adapted product charts are real analytic. Hence
\[
\gamma_{ij}^*h_i=e^{2\sigma_{ij}}h_j
\]
for a real-analytic function $\sigma_{ij}$; indeed,
$e^{2\sigma_{ij}}=q^{-1}\operatorname{tr}_{h_j}(\gamma_{ij}^*h_i)$. Setting
\[
N:=\coprod_i N_i,
\qquad
h|_{N_i}:=h_i,
\]
we obtain a real-analytic Riemannian metric $h$ on $N$ for which the data
$\{U_i,\pi_i,\gamma_{ij}\}$ satisfy Definition~\ref{def:transversely-analytic}.
Thus $\F_1$ is transversely real-analytic.
\end{proof}

Tarquini asked whether every conformal foliation on a compact manifold is
necessarily Riemannian or transversely M\"obius, and proved the assertion
under a transverse real-analyticity assumption \cite[Theorem~0.0.1]{Tarquini2004}.
We state the analogous statement in the smooth category as the following conjecture.

\begin{conjecture}[Smooth dichotomy for conformal foliations]\label{conj:foliation}
Let $\F$ be a conformal foliation of codimension $q\geq3$ on a compact,
connected smooth manifold. Then $\F$ is either Riemannian or transversely
M\"obius, equivalently, a $(\Conf(S^q),S^q)$-foliation.
\end{conjecture}

If this conjecture holds, the analyticity assumption in Theorem~\ref{thm:analytic-main} can be removed in case of $\max \{\dim T_1,\dim T_2\}\ge 3$. This motivates the following conjecture.

\begin{conjecture}[Smooth globalization]\label{conj:main}
Let $(M,g)$ be compact, connected, and simply connected, with dimension $n\geq5$. If a conformal product structure $TM=T_1\oplus T_2$ on $M$ satisfies $\max\{\dim T_1,\dim T_2\}\ge 3$, then it must be induced by a global product splitting.
\end{conjecture}

\begin{remark}
The full real-analyticity assumption in Theorem~\ref{thm:analytic-main} is
stronger than necessary. The proof only uses that the transverse conformal
structure of $\F_1$ induced by $[g|_{T_2}]$ is real analytic.
\end{remark}

We also recall the following decomposition theorem of Ponge and Reckziegel
\cite[Theorem~1]{PongeReckziegel1993}.

\begin{theorem}[Ponge--Reckziegel]\label{thm:Ponge-Reckziegel}
Let $(M,h)$ be a simply connected, geodesically complete pseudo-Riemannian manifold with two complementary, orthogonal foliations $\mathcal L$ and $\mathcal K$. If $\mathcal L$ is totally geodesic and $\mathcal K$ is totally umbilical, then $(M,h)$ is isometric to a twisted product whose canonical factor foliations correspond to $\mathcal L$ and $\mathcal K$.
\end{theorem}

In the next two sections we consider separately the two cases given by \ref{thm:Tarquini}.

\section{The Riemannian case}

The arguments in this section hold in the smooth category.

Let $\F$ be a conformal foliation of codimension $q$ with transverse conformal class $[h]$ on a simply connected Riemannian manifold $(M,g)$. Assume that $\F$ is Riemannian in the sense of Definition~\ref{def:Riemannian-foliation}, which means that there exists a holonomy-invariant transverse metric. We first establish the following lemma:

\begin{lemma}\label{lem:invariant-conformal-representative}
The conformal class $[h]$ contains a holonomy-invariant Riemannian metric.
\end{lemma}

\begin{proof}
Let $k$ be a holonomy-invariant transverse Riemannian metric, and let $h$ be a representative of the given transverse conformal class. Both $h$ and $k$ are Riemannian metrics on $N$, so one can compare their volume elements by:
\[
r:=\frac{|d\mathrm{vol}_k|}{|d\mathrm{vol}_h|}>0.
\]
For every transverse change of coordinates $\gamma$ and $\sigma$ such that $\gamma^*h=e^{2\sigma}h$. Since $\gamma^*k=k$, we obtain
\[
\gamma^*(d\mathrm{vol}_k)=d\mathrm{vol}_k,
\qquad
\gamma^*(d\mathrm{vol}_h)
=e^{q\sigma}d\mathrm{vol}_h.
\]
Consequently,
\[
r\,d\mathrm{vol}_h=d\mathrm{vol}_k
=\gamma^*(d\mathrm{vol}_k)=\gamma^*(r\,d\mathrm{vol}_h)=(r\circ\gamma)\gamma^*(d\mathrm{vol}_h)=(r\circ\gamma)e^{q\sigma}d\mathrm{vol}_h,
\]
which gives
\[
(r\circ\gamma)e^{q\sigma}=r.
\]

Define $\eta:=r^{2/q}h$. Then $\eta\in[h]$, and
\[
\begin{aligned}
\gamma^*\eta
&=\gamma^*(r^{2/q}h)\\
&=(r\circ\gamma)^{2/q}\gamma^*h\\
&=(re^{-q\sigma})^{2/q}e^{2\sigma}h\\
&=r^{2/q}h=\eta.
\end{aligned}
\]
Thus $\eta$ is a holonomy-invariant Riemannian metric
in $[h]$.
\end{proof}

\begin{lemma}[Bundle-like consequence]\label{lem:bundle-like}
Assume that $\F$ is Riemannian in the sense of
Definition~\ref{def:Riemannian-foliation}, and fix an $N$-cocycle and a metric
$h_N$ satisfying that definition. Let $H\subset TM$ be a smooth subbundle with
$TM=T\F\oplus H$. For $x\in U_i$ and $u,v\in H_x$, set
\[
(h_H)_x(u,v):=(h_N)_{f_i(x)}\bigl(df_i|_x(u),df_i|_x(v)\bigr).
\]
Then $h_H$ is a well-defined Riemannian metric on $H$. Choose any Riemannian
metric $g_{\F}$ on $T\F$ and define a Riemannian metric $\widehat g$ on $M$ by
$T\F\perp_{\widehat g}H$, $\widehat g|_{T\F}=g_{\F}$, and
$\widehat g|_H=h_H$.

For every $i$, let $\bar U,\bar V$ be vector fields on $f_i(U_i)$ and let
$U,V\in\Gamma(H|_{U_i})$ be their unique lifts determined by
$df_i(U)=\bar U\circ f_i$ and $df_i(V)=\bar V\circ f_i$. Then, for every
$X\in\Gamma(T\F|_{U_i})$,
$(\mathcal L_X\widehat g)(U,V)=0$.
\end{lemma}

\begin{proof}
For $x\in U_i\cap U_j$, let $W$ be the connected component of
$U_i\cap U_j$ containing $x$, and denote the corresponding transverse change
of coordinates by $\gamma_{ij}$. Since
$f_i|_W=\gamma_{ij}\circ f_j|_W$ and
$\gamma_{ij}^*h_N=h_N$, for $u,v\in H_x$ we have
\[
\begin{aligned}
(h_N)_{f_i(x)}(df_i u,df_i v)
&=(h_N)_{\gamma_{ij}(f_j(x))}
   (d\gamma_{ij}(df_j u),d\gamma_{ij}(df_j v))\\
&=(h_N)_{f_j(x)}(df_j u,df_j v).
\end{aligned}
\]
Thus the formula defining $h_H$ is independent of $i$. Since
$\ker(df_i)=T\F|_{U_i}$ and $TM=T\F\oplus H$, the map
$df_i|_{H_x}:H_x\to T_{f_i(x)}N$ is an isomorphism for every $x\in U_i$;
therefore the lifts $U,V$ in the statement exist uniquely and are smooth on $U_i$.

The field $X$ is $f_i$-related to the zero field on $f_i(U_i)$, while $U$ and
$V$ are $f_i$-related to $\bar U$ and $\bar V$. Hence the naturality of Lie
brackets under a submersion gives
$df_i([X,U])=[0,\bar U]\circ f_i=0$ and
$df_i([X,V])=[0,\bar V]\circ f_i=0$. Thus
$[X,U],[X,V]\in\Gamma(T\F|_{U_i})$. Moreover,
$\widehat g(U,V)=h_N(\bar U,\bar V)\circ f_i$, and therefore
$X(\widehat g(U,V))=0$ because $df_i(X)=0$. Since
$T\F\perp_{\widehat g}H$,
\[
(\mathcal L_X\widehat g)(U,V)
=X(\widehat g(U,V))-\widehat g([X,U],V)-\widehat g(U,[X,V])=0.
\]
\end{proof}

\begin{proposition}\label{prop:Riemannian}
If $\F_1$ is a Riemannian foliation, then $M$ is globally
diffeomorphic to a product whose factor distributions are
$T_1$ and $T_2$.
\end{proposition}

\begin{proof}
By Lemma~\ref{lem:invariant-conformal-representative}, the transverse conformal class
induced by $[g|_{T_2}]$ contains a holonomy-invariant Riemannian metric
$\eta$. Thus $\eta=e^{2u}g|_{T_2}$ for some $u\in C^\infty(M)$. Replacing
$g$ by $e^{2u}g$ and retaining the notation $g$, we may assume that
$g|_{T_2}=\eta$. In Lemma~\ref{lem:bundle-like}, take $H=T_2$ and choose
$g_{\F_1}=g|_{T_1}$, so that $\widehat g=g$. Hence, on each foliation chart
$U_i$, for every $X\in\Gamma(T_1|_{U_i})$ and every pair
$U,V\in\Gamma(T_2|_{U_i})$ obtained as the lifts of vector fields on
$f_i(U_i)$, one has
\begin{equation}\label{bundle like}
(\mathcal L_Xg)(U,V)=0.
\end{equation}

As $T_2$ is integrable, $[U,V]\in\Gamma(T_2)$. The Koszul formula and \eqref{bundle like} give
\[
\begin{aligned}
2g(\nabla_UV,X)
&=-Xg(U,V)+g([X,U],V)+g(U,[X,V])+g([U,V],X)\\
&=-(\mathcal L_Xg)(U,V)=0.
\end{aligned}
\]
Hence $\F_2$ is totally geodesic.

By \cite[Proposition~2.4]{Jiang2026}, $\F_1$ is totally umbilical with respect
to $g$. Since $M$ is compact, $(M,g)$ is geodesically complete. Applying
Theorem~\ref{thm:Ponge-Reckziegel} to $\F_2$ and $\F_1$ gives a global twisted
product decomposition.
\end{proof}

\section{The M\"obius case}

We now consider the second case in Theorem \ref{thm:Tarquini}. Let $\mathcal{F}$ be a transversely M\"obius foliation, equivalently, a $(\text{Conf}(S^q),S^q)$-foliation. We first obtain a globalization statement that will be used below.

\begin{lemma}\label{lem:flat-connection}
Let $p:E\to B$ be a proper surjective submersion between connected smooth
manifolds, with connected fibers and simply connected base $B$. Denote by
$V:=\ker(dp)$ and suppose that $TE=V\oplus H$ for a smooth integrable distribution $H$. Then, for every $b_0\in B$, there is a
diffeomorphism $\Phi:p^{-1}(b_0)\times B\xrightarrow{\cong}E$ which sends the tangent bundle of the first factor to $V$ and that of the second
factor to $H$.
\end{lemma}
\begin{proof}
Let $L_x$ denote the leaf of $H$ through $x\in E$. Since
\[
T_yL_x=H_y,
\qquad
dp|_{H_y}:H_y\xrightarrow{\cong}T_{p(y)}B,
\]
the restriction
\[
p_x:=p|_{L_x}:L_x\longrightarrow B
\]
is a local diffeomorphism.

Let $\gamma:[0,1]\to B$ be a smooth path. A horizontal lift
$\widetilde\gamma$ of $\gamma$ satisfies
\begin{equation}\label{3.1}
p(\widetilde\gamma(t))=\gamma(t),
\qquad
\dot{\widetilde\gamma}(t)
=
\bigl(dp|_{H_{\widetilde\gamma(t)}}\bigr)^{-1}
\bigl(\dot\gamma(t)\bigr).
\end{equation}
For every prescribed $e_0\in p^{-1}(\gamma(0))$, Equation \eqref{3.1} has a unique local solution with $\widetilde\gamma(0)=e_0$. Moreover,
\[
\widetilde\gamma([0,a))
\subset
p^{-1}\bigl(\gamma([0,1])\bigr)
\]
for every interval $[0,a)$ on which the solution is defined. Since $p$ is proper, the set on the right is compact; hence the solution cannot cease to exist in finite time. Thus every path in $B$ admits a unique horizontal lift on $[0,1]$ through any prescribed initial point.

Because $H$ is integrable, a horizontal lift starting in $L_x$ remains
in $L_x$. Hence $p_x$ has the path-lifting property. Since $B$ is
connected, lifting paths starting at $p(x)$ shows that $p_x$ is surjective. Therefore $p_x:L_x\longrightarrow B$ is a surjective local diffeomorphism with the path-lifting property, hence a covering map. Since $L_x$ is connected and $B$ is simply connected,
\[
p_x:L_x\xrightarrow{\cong}B
\tag{3.2}
\]
is a diffeomorphism.

Fix $b_0\in B$ and put $P:=p^{-1}(b_0)$. Every $H$-leaf meets every fiber of $p$ in exactly one point. Hence
\[
\Phi:P\times B\longrightarrow E,
\qquad
\Phi(x,b):=p_x^{-1}(b),
\tag{3.3}
\]
is well defined and bijective.

It remains to check smoothness. Fix $b\in B$, choose a smooth path
$\alpha$ from $b_0$ to $b$, and let $U\ni b$ be a coordinate ball.
For $b'\in U$, choose the radial path in $U$ from $b$ to $b'$. Parallel
transport first along $\alpha$ and then along this radial path gives
\[
P\times U\longrightarrow p^{-1}(U).
\]
By smooth dependence of the solutions of \eqref{3.1} on the initial
point and on parameters, this map is smooth, it is
precisely $\Phi|_{P\times U}$. Reversing the paths gives a smooth local
inverse. Hence
\[
\Phi:P\times B\xrightarrow{\cong}E
\]
is a diffeomorphism.

Finally,
\[
p\circ\Phi=\operatorname{pr}_B.
\]
For fixed $b$, $\Phi(\,\cdot\,,b)$ takes values in the fiber
$p^{-1}(b)$, whereas for fixed $x$, $\Phi(x,\cdot)=p_x^{-1}$ takes
values in the leaf $L_x$. Therefore
\[
d\Phi(TP)=V,
\qquad
d\Phi(TB)=H.
\]
\end{proof}

\begin{remark}
The proof is analogous to the corresponding statement of Kobayashi--Nomizu \cite[Chapter~II, \S9, Corollary~9.2]{KobayashiNomizu1963} for principal bundles.
\end{remark}

\begin{proposition}\label{prop:spherical}
Suppose that $\F_1$ is a $(\Conf(S^q),S^q)$-foliation, where $q\geq3$. Then $M$ is globally diffeomorphic to a product whose factor distributions are
$T_1$ and $T_2$.
\end{proposition}

\begin{proof}
Write the transverse structure as an $S^q$-cocycle $\{U_i,f_i,\gamma_{ij}\}$ as in Definition~\ref{def:spherical-foliation}. On each connected component of an overlap, $\gamma_{ij}$ is the restriction of an element of $\Conf(S^q)$. Fix $x_0\in M$ and one of the local submersions $f_{i_0}$ near $x_0$. For $x\in M$, choose a path from $x_0$ to $x$, cover it by a finite chain of the sets $U_i$, and successively use the transition elements in $\text{Conf}(S^q)$ to continue $f_{i_0}$ along the chain. The cocycle identities show that the resulting value at $x$ is unchanged by refining the chain. If two paths have the same endpoints, they are homotopic relative to their endpoints because $M$ is simply connected. Subdividing a homotopy square so that the image of each small rectangle lies in one of the sets $U_i$ , and using the cocycle identities on adjacent rectangles, shows that the two continuations agree. Hence the continuation is independent of the chosen path and defines a global map $\Gamma: M\to S^q$. By construction, $\Gamma$ is a smooth submersion and $\ker(d\Gamma)=T\widetilde{\F}_1$. 

The image of a submersion is open. On the other hand, $M$ is compact, so $\Gamma(M)$ is compact and therefore closed in $S^q$. Since $S^q$ is connected and the image is nonempty, $\Gamma(M)=S^q$. Thus $\Gamma$ is a proper surjective submersion. By Ehresmann's fibration theorem, it is a smooth locally trivial fiber bundle.

We now show that its fibers are connected. Fix $s_0\in S^q$ and a connected component $P_0$ of $\Gamma^{-1}(s_0)$. Continue $P_0$ through bundle trivializations along a path from $s_0$ to $s\in S^q$; this produces a connected component $P_s\subset\Gamma^{-1}(s)$. The resulting component is independent of the chosen path. Indeed, two such paths are homotopic relative to their endpoints because $S^q$ is simply connected, and after subdividing a homotopy square so that the image of each small rectangle lies in a bundle-trivializing neighborhood, the component transported along opposite sides of each rectangle is the same. Set $M_0:=\bigcup_{s\in S^q}P_s$. In every bundle trivialization, $M_0$ is the product of the base neighborhood with one connected component of the model fiber. Since components of a manifold are both open and closed, $M_0$ is both open and closed in $M$. Because $M$ is connected, $M_0=M$. Intersecting with $\Gamma^{-1}(s_0)$ gives $P_0=\Gamma^{-1}(s_0)$, so this fiber is connected; the same argument applies to every fiber. Consequently each fiber of $\Gamma$ is exactly one leaf of $\mathcal{F}_1$.

Finally, the original conformal product splitting gives $TM=T_1\oplus T_2$, and $T_2$ is integrable. Since
$\ker(d\Gamma)=T_1$, the restriction
$d\Gamma|_{T_2}:T_2\to TS^q$ is a vector-bundle isomorphism; thus
$T_2$ is precisely an integrable Ehresmann complement for
$\Gamma$. Applying Lemma~\ref{lem:flat-connection} to
$p=\Gamma:M\to S^q$, $V=T_1$ and $H=T_2$, we obtain, for any
$s_0\in S^q$, a diffeomorphism
$\Gamma^{-1}(s_0)\times S^q\xrightarrow{\cong}M$ whose two factor
tangent bundles are $T_1$ and $T_2$.
\end{proof}

\begin{proof}[Proof of Theorem~\ref{thm:analytic-main}]
After interchanging the two distributions if necessary, put
$q:=\rk T_2\geq3$. By Lemma~\ref{lem:conformal-foliation}, $\F_1$ is a
conformal foliation of codimension $q$, and by
Lemma~\ref{lem:analytic-foliation} its transverse conformal structure is real
analytic. Since $M$ is compact and connected, Theorem~\ref{thm:Tarquini}
applies. In the Riemannian case,
Proposition~\ref{prop:Riemannian} gives the global product splitting. In the
transversely M\"obius case, $\F_1$ is a
$(\Conf(S^q),S^q)$-foliation, and Proposition~\ref{prop:spherical} gives the
same conclusion.
\end{proof}
This concludes the proof of Theorem \ref{thm:analytic-main} in the case of $\max\{\dim T_1,\dim T_2\}\ge3$.

\section{Remaining low dimension cases}

We next treat the only four-dimensional case not covered by the preceding Tarquini argument, namely $\dim T_1=\dim T_2=2$.

\begin{proposition}\label{prop:dimension-four}
Let $(M,c)$ be a conformal manifold of dimension $n=4$ which is compact and simply connected. If a conformal product structure $TM=T_1\oplus T_2$ is $c$-orthogonal splitting with $\rk T_1=\rk T_2=2$, then its splitting is a global product splitting.
\end{proposition}
\begin{remark}
No analyticity assumption is needed in this proposition.
\end{remark}
We first establish the following Lemma of existence of Hermitian structure.
\begin{lemma}\label{lem:parallel-Hermitian}
Assume that $T_1$ and $T_2$ are oriented. Then there exists a $D$-parallel almost Hermitian structure $J$ compatible with $c$. More precisely,
\[
        J^2=-\mathrm{Id},\qquad c(JX,JY)=c(X,Y),\qquad DJ=0.
\]
Moreover, $J$ is integrable.
\end{lemma}
\begin{proof}
We construct $J$ directly from the two oriented conformal $2$-plane bundles $T_1$ and $T_2$. Fix $i\in\{1,2\}$. Choose locally a metric $g\in c$, with corresponding scale $\ell=\ell_g$, so that
\[
        c(X,Y)=g(X,Y)\ell^2.
\]
Let $(e_i,f_i)$ be a local positively oriented $g$-orthonormal frame of
$T_i$. Define a section $\omega_i\in \Gamma(\Lambda^2 T_i^*\otimes L^2)$ by $\omega_i:=e_i^\flat\wedge f_i^\flat\otimes \ell^2$, where $\flat$ is taken with respect to $g$.

This definition is independent of the choice of the positively oriented
$g$-orthonormal frame, since another such frame differs by an element of
$SO(2)$ and hence has determinant $1$.

It is also independent of the choice of $g\in c$. Indeed, if $g'=e^{2u}g$, then the corresponding scale is $\ell_{g'}=e^{-u}\ell$. If $(e_i,f_i)$ is $g$-orthonormal, then $(e_i',f_i')=(e^{-u}e_i,e^{-u}f_i)$ is $g'$-orthonormal, and the $g'$-dual one-forms satisfy
\[
        (e_i')^{\flat_{g'}}=e^u e_i^\flat,\qquad
        (f_i')^{\flat_{g'}}=e^u f_i^\flat.
\]
Therefore
\[
        (e_i')^{\flat_{g'}}\wedge (f_i')^{\flat_{g'}}
        \otimes \ell_{g'}^2
        =
        e^{2u} e_i^\flat\wedge f_i^\flat
        \otimes e^{-2u}\ell^2
        =
        e_i^\flat\wedge f_i^\flat\otimes \ell^2.
\]
Hence $\omega_i$ is a well-defined weightless volume form of $T_i$.

Extend $\omega_i$ to all of $TM$ by setting $\omega_i(X,Y):=\omega_i(\pi_iX,\pi_iY)$, where $\pi_i:TM\to T_i$ is the projection along the splitting $TM=T_1\oplus T_2$. Since the splitting is $D$-parallel, the projections $\pi_i$ are $D$-parallel. We now check that $D\omega_i=0$. Let $D\ell=\theta\otimes\ell$ locally. Since $D$ preserves $T_i$, we may
write, for a vector field $X$,
\[
        D_Xe_i=a(X)e_i+b(X)f_i,\qquad
        D_Xf_i=\alpha(X)e_i+\beta(X)f_i.
\]
Because $Dc=0$ and
\[
c(e_i,e_i)=c(f_i,f_i)=\ell^2,\qquad c(e_i,f_i)=0,
\]
we obtain
\[
a=\theta,\qquad \beta=\theta,\qquad \alpha=-b.
\]
Thus
\[
D_Xe_i=\theta(X)e_i+b(X)f_i,\qquad D_Xf_i=-b(X)e_i+\theta(X)f_i.
\]
Since $e_i^\flat=g(e_i,\cdot)$ and $Dg=-2\theta\otimes g$, for every vector field $Y$ we have
\[
\begin{aligned}
(D_Xe_i^\flat)(Y)
&=X\bigl(g(e_i,Y)\bigr)-g(e_i,D_XY)\\
&=(D_Xg)(e_i,Y)+g(D_Xe_i,Y)\\
&=-2\theta(X)g(e_i,Y)
  +g\bigl(\theta(X)e_i+b(X)f_i,Y\bigr)\\
&=-\theta(X)e_i^\flat(Y)+b(X)f_i^\flat(Y).
\end{aligned}
\]
Hence
\[
D_Xe_i^\flat
=
-\theta(X)e_i^\flat+b(X)f_i^\flat.
\]
Similarly,
\[
\begin{aligned}
(D_Xf_i^\flat)(Y)
&=(D_Xg)(f_i,Y)+g(D_Xf_i,Y)\\
&=-2\theta(X)g(f_i,Y)
  +g\bigl(-b(X)e_i+\theta(X)f_i,Y\bigr)\\
&=-b(X)e_i^\flat(Y)-\theta(X)f_i^\flat(Y),
\end{aligned}
\]
and therefore
\[
D_Xf_i^\flat
=
-b(X)e_i^\flat-\theta(X)f_i^\flat.
\]
Therefore
\[
\begin{aligned}
        D_X(e_i^\flat\wedge f_i^\flat)&=[-\theta(X)e_i^\flat+b(X)f_i^\flat]\wedge f_i^\flat+e_i^\flat\wedge[-b(X)e_i^\flat-\theta(X)f_i^\flat]\\
        &=-2\theta(X)e_i^\flat\wedge f_i^\flat.
\end{aligned}
\]
On the other hand, $D_X(\ell^2)=2\theta(X)\ell^2$. Consequently,
\[
D_X(e_i^\flat\wedge f_i^\flat\otimes \ell^2)=0,
\]
so $D\omega_i=0$. Now define $J_i\in\Gamma(\operatorname{End}(TM))$ by:
\[
c(J_iX,Y)=\omega_i(X,Y).
\]
This is well-defined because, for fixed $X$, both sides take values in $L^2$,
and the map
\[
        TM\longrightarrow T^*M\otimes L^2,\qquad
        V\longmapsto c(V,\cdot)
\]
is an isomorphism. In the frame $(e_i,f_i)$ we have
\[
        J_ie_i=f_i,\qquad J_if_i=-e_i,
\]
and $J_i$ vanishes on $T_j$ for $j\neq i$. Hence
\[
        J_i^2=-\pi_i,\qquad J_1J_2=J_2J_1=0.
\]
Set $J:=J_1+J_2$. Then
\[
J^2=J_1^2+J_2^2=-(\pi_1+\pi_2)=-\mathrm{Id},
\]
so $J$ is an almost complex structure. It is compatible with $c$. Indeed,
\[
c(JX,Y)=\omega_1(X,Y)+\omega_2(X,Y)
\]
is skew-symmetric, and hence
\[
c(JX,Y)=-c(X,JY),\qquad
c(JX,JY)=-c(X,J^2Y)=c(X,Y).
\]
Finally, since $Dc=0$ and $D\omega_i=0$, the identity
$c(J_iX,Y)=\omega_i(X,Y)$ implies $DJ_i=0$, and therefore $DJ=0$. The
integrability of $J$ follows from \cite[Lemma~5.6]{BelgunMoroianu2011}. Thus
$(M,c,J,D)$ is a Hermitian Weyl manifold.
\end{proof}

\begin{proof}[Proof of Proposition~\ref{prop:dimension-four}]
Since $M$ is simply connected, both distributions are orientable. Choose orientations on $T_1$ and $T_2$. By Lemma~\ref{lem:parallel-Hermitian}, there is an integrable complex structure $J$ compatible with $c$ and satisfying $DJ=0$.

For $i\in\{1,2\}$, set
\[
L_i:=(T_i\otimes\mathbb C)\cap T^{1,0}M.
\]
Then
\[
T^{1,0}M=L_1\oplus L_2.
\]
This splitting is holomorphic. Indeed, let
$Z\in\Gamma(T^{0,1}M)$ and $s\in\Gamma(L_i)$. Since $DJ=0$, the connection
$D$ preserves the type decomposition, and since $H_i$ is $D$-parallel,
\[
D_Zs\in\Gamma(L_i),
\qquad
D_sZ\in\Gamma(T^{0,1}M).
\]
Since $D$ is torsion-free, we obtain
\[
[Z,s]^{1,0}=(D_Zs-D_sZ)^{1,0}=D_Zs\in\Gamma(L_i).
\]
Thus the Dolbeault operator of $T^{1,0}M$ preserves each $L_i$, so $L_1$ and $L_2$ are holomorphic line subbundles. Beauville's classification \cite[Theorem~C]{Beauville2000} now applies. The Hopf-surface case is impossible because $M$ is simply connected. Therefore the universal cover of $M$ is biholomorphic to a product $\Sigma_1\times\Sigma_2$ of simply connected Riemann surfaces, and the lifted splitting is the canonical splitting $T^{1,0}\Sigma_1\oplus T^{1,0}\Sigma_2$. As $M$ is compact and simply connected, each $\Sigma_i$ is compact and simply connected, hence $\Sigma_i\simeq S^2$. Thus the original splitting is a global product splitting.
\end{proof}

Finally, suppose that $\dim M=3$. One of the two distributions has rank $2$ and, being real-analytic and involutive, defines a real-analytic foliation of codimension one. Since $M$ is simply connected, this contradicts Haefliger's
non-existence theorem \cite[\S~4.3]{Haefliger1962}, which claims that a compact real-analytic manifold whose fundamental group consists entirely of elements of finite order admits no real-analytic foliation of codimension one. This concludes the proof of Theorem \ref{thm:analytic-main}.

\textbf{Declaration on the use of AI:} 
The author used ChatGPT (OpenAI) to assist with literature searches, mathematical discussions and LaTeX editing. The author takes full responsibility for all mathematical statements, proofs, and conclusions in this paper.

\printbibliography

\end{document}